\documentclass{article}
\usepackage[utf8]{inputenc}
\usepackage{titlesec}

\usepackage{amsthm}
\usepackage{amssymb}
\usepackage{enumerate}
\usepackage{tikz-cd}
\usepackage{mathtools}
\usepackage{cleveref}
\usepackage[margin=1.5in]{geometry}
\tikzcdset{arrow style=math font}
\usepackage{braket}

\titleformat{\section}
  {\normalfont\fontsize{11}{12}\bfseries}{\thesection}{1em}{}

\titleformat{\subsection}[runin]
  {\normalfont\fontsize{10}{11}\bfseries}{\thesubsection}{.5em}{}
  
\theoremstyle{plain}
  \newtheorem{thm}{Theorem}[section]
  \newtheorem{cor}[thm]{Corollary}
  \newtheorem{lemma}[thm]{Lemma}
  \newtheorem{prop}[thm]{Proposition}
  
\theoremstyle{definition}
\newtheorem{dfn}[thm]{Definition}

\newtheorem{rmk}[thm]{Remark}

\newtheorem{ex}[thm]{Example}

\numberwithin{equation}{section}

\newcommand{\f}{\mathcal F}
\newcommand{\p}{\mathcal P}
\newcommand{\e}{\mathcal E}
\renewcommand{\hom}{\mathcal{H}om}
\newcommand{\ext}{\mathcal{E}xt_f}
\renewcommand{\dot}{{\boldsymbol{\cdot}}}
\newcommand{\M}{{M_1 \times M_2}}
\newcommand{\bl}{{Bl_{\Gamma}(\M)}}
\newcommand{\m}{M_\sigma(2\lambda_1+\lambda_2)}

\renewcommand{\o}{{\mathcal O}}
\newcommand{\K}{K_{top}(Ku(Y))}
\renewcommand{\k}{K_{num}(Ku(Y))}

\title{The Voisin Map via Families of Extensions}
\author{Huachen Chen}
\date{}

\begin{document}
\maketitle

\begin{abstract}
 \noindent We prove that given a cubic fourfold $Y$ not containing any plane, the Voisin map $v: F(Y)\times F(Y) \dashrightarrow Z(Y)$ constructed in \cite{Voi}, where $F(Y)$ is the variety of lines and $Z(Y)$ is the Lehn-Lehn-Sorger-van Straten eightfold \cite{LLSS}, can be resolved by blowing up the incident locus $\Gamma \subset F(Y)\times F(Y)$ endowed with the reduced scheme structure. Moreover, if $Y$ is very general, then this blowup is a relative Quot scheme over $Z(Y)$ parametrizing quotients in a heart of a Kuznetsov component of $Y.$ 
\end{abstract}

\section{Introduction}

The derived category $D^b(Y)$ of a smooth cubic fourfold $Y$ decomposes $$D^b(Y)=\langle Ku(Y), \mathcal O_Y, \mathcal O_Y(1), \mathcal O_Y(2)\rangle$$ into a relatively simple part (an exceptional collection) and a highly nontrivial subcategory $Ku(Y)$, the Kuznetsov component, which is believed to encode the geometry of $Y$ \cite{Kuz4}. One example of this principle is the work of \cite{LLMS} and \cite{LPZ}, which shows that the variety of lines $F(Y)$ and the LLSvS eightfold $Z(Y)$ constructed in \cite{LLSS} (with the assumption that $Y$ does not contain any plane) both can be naturally interpreted as moduli of stable objects in $Ku(Y),$ with respect to a Bridgeland stability condition on $Ku(Y)$ constructed in \cite{BLMS}.  

In \cite{Voi}, C. Voisin established, in particular, a degree six rational map $v: F(Y)\times F(Y) \dashrightarrow Z(Y)$ using geometry of the cubic fourfold. In the light of \cite{LLMS} and \cite{LPZ}, one can reinterpret the Voisin map via families of extensions in $Ku(Y).$ Indeed, there are two moduli spaces $M_\sigma(\lambda_1)$ and $M_\sigma(\lambda_1+\lambda_2)$ that are both isomorphic to $F(Y),$ and $M_\sigma(2\lambda_1+\lambda_2)$ is isomorphic to $Z(Y),$ where $\lambda_i$ are numerical classes (see \cref{sec2.1}) and $\sigma$ is a Bridgeland stability condition on $Ku(Y)$ (see \cite{BLMS} and \cref{sec2.2}), such that for a general point $(F, P)\in M_\sigma(\lambda_1) \times M_\sigma(\lambda_1+\lambda_2),$ one has $Ext^1(P, F)\cong \mathbb C.$ The unique nontrivial extension gives a stable object of class $2\lambda_1+\lambda_2$ and therefore defines $$v: M_\sigma(\lambda_1) \times M_\sigma(\lambda_1+\lambda_2) \dashrightarrow \m,$$ which coincides with Voisin's construction.

The purpose of this note is to study the Voisin map by addressing how the family of extensions spread along the indeterminacy locus. 

 Let $\Gamma$ be the incident locus $\{(L_1, L_2)\in F(Y)\times F(Y): L_1\cap L_2 \neq \emptyset \}.$ As we will see, if $(F, P)\in \Gamma,$ identifying $F(Y)\times F(Y)$ with $M_\sigma(\lambda_1)\times M_\sigma(\lambda_1+\lambda_2),$ then $ext^1(P, F)>1.$ Thus, the map $v$ is not defined on $\Gamma$ (we notice that this agrees with a result of \cite{Mur}). Denoted by $\mathcal F$ and $\p$ the pullbacks of two universal families on $M_\sigma(\lambda_1)\times Y$ and $ M_\sigma(\lambda_1+\lambda_2)\times Y$, respectively, to $M_\sigma(\lambda_1)\times M_\sigma(\lambda_1+\lambda_2) \times Y.$ Let $\mathcal A$ be the heart of the t-structure associated to $\sigma.$ By considering a notion of families of extensions of $\p$ by $\f$ as in \cite{lange} (see \cref{foe}), we prove:

\begin{thm}[c.f. \cref{mainthm}]\label{thm1} Let $b: Bl_\Gamma (M_\sigma(\lambda_1)\times M_\sigma(\lambda_1+\lambda_2)) \to M_\sigma(\lambda_1)\times M_\sigma(\lambda_1+\lambda_2)$ be the blowup along the reduced scheme structure of $\Gamma$.
\begin{enumerate}[(a).]
    \item Suppose that $Y$ does not contain any plane. The Voisin map can be resolved by $b$:
{\center
\begin{tikzcd}[column sep=tiny]
 & Bl_{\Gamma}(M_\sigma(\lambda_1) \times M_\sigma(\lambda_1+\lambda_2))  \arrow[ld, "b"]{dashed} \arrow[rd, "q"]\\
M_\sigma(\lambda_1)\times M_\sigma(\lambda_1+\lambda_2) \arrow[rr, dashed, "v"]& & M_{\sigma}(2\lambda_1+\lambda_2)
      
\end{tikzcd}\par}

\item \label{quot} Moreover, if $Y$ is very general, then the blowup above is a relative Quot scheme over $\m$ parametrizing quotients in the heart $\mathcal A$ and of class $\lambda_1+\lambda_2.$
\end{enumerate}

\end{thm}

The main idea is to consider a functor of families of non-splitting extensions of $\p$ by $\f$. Let $f: \M \times Y \to \M$ be the projection, define $R\hom_f:=Rf_*R\hom,$ and $\ext^i:= \mathcal H^i(R\hom_f).$ In \cite{lange}, it has been shown that the set of families of extensions (\cref{foe}) over a reduced scheme $g:T\to S$ is the same as $H^0(T, g^*\ext^1(\p, \f)),$ provided that $\ext^1(\p,\f)$ satisfies a condition called "commute with base change" (\cref{cwbc}). However, such a condition does not hold in our case. This can be fixed; indeed the right functor to consider should be $$(g:T\to S) \to Hom(\o_T, Lg^*R\hom_f(\p,\f)[1]),$$  which renders the condition "commute with base change" automatic. Meanwhile, this functor is still a sheaf on $S$ in our case, and consequently we have a bijection between $Hom(\o_T, Lg^*R\hom_f(\p,\f)[1])$ and the set of families of extensions of $\p_T$ by $\f_T$ (\cref{psi1}).

Using this observation, we can adapt the functor of families of non-splitting extensions to our case (\cref{psi2}), and then (in \cref{sec4}) show that it is represented by the blowup in \cref{thm1}, and obtain a morphism to $Z(Y)$. 


Once we prove part (a), the missing ingredient for part (b) is a universal relative quotient on the blowup. To obtain that, we take a global family of extensions on the blowup which splits along the exceptional divisor, and perform an elementary modification to get a nowhere-split global family. See \cref{sec4}.

\subsection*{Acknowledgements} The idea of studying the Voisin map using moduli and relative Quot schemes belongs to Arend Bayer and Aaron Bertram. I am grateful to them for leaving the question to me. Also, many thanks to Sukhendu Mehrotra, Alex Perry, Paolo Stellari, Franco Rota and Xiaolei Zhao for helpful discussions. Most of these discussions happened during a wonderful CIMI/FRG workshop in Toulouse, I would like to thank the organizers Marcello Bernardara and Emanuele Macr\`i.

\section{Review on Kuznetsov components, stability conditions and constant families of t-structures}

In this section, we first recall some results about Kuznetsov components of cubic fourfolds, Bridgeland stability conditions and moduli of stable objects that are crucial for us. Then in \cref{sec2.3} we put together \cite{Kuz} and \cite{AP}, \cite{Po} to clarify a slight technical issue, namely, $\f, \p$ in our case are not families of sheaves but complexes, and we would like to consider their extensions lie in a certain (sheaf of) hearts of t-structures.

\subsection{Kuznetsov components of cubic fourfolds.}\label{sec2.1} Let $Y$ be a smooth cubic fourfold, and $D^b(Y)$ be the bounded derived category of coherent sheaves on $Y.$ There is an exceptional collection $\langle \mathcal O_Y, \mathcal O_Y(1),\mathcal O_Y(2)\rangle$ of $D^b(Y)$ (see e.g. \cite{Kuz4}).

\begin{dfn}
The Kuznetsov component $Ku(Y)$ of $Y$ is defined as $$Ku(Y):=\{E\in D^b(X): R^\dot Hom(\mathcal O_Y(i), E)=0,\ i=0,1,2\}.$$
\end{dfn}

\noindent It is an admissible triangulated subcategory of $D^b(Y),$ i.e., the inclusion $i: Ku(Y)\hookrightarrow D^b(Y)$ admits a left (and right) adjoint $i^*: D^b(Y)\to Ku(Y)$ (and $i^!$). The Serre functor on $Ku(Y)$ is $(-)[2]$ \cite{Kuz2}.

 Let $K_{top}(Y)$ be the topological K-theory of $Y,$ $\chi(-,-)$ be the Euler pairing. We recall the definition of Mukai lattices of Kuznetsov components by Addington and Thomas \cite{AT}, see also \cite{BLMS}.
\begin{dfn} 
The topological Mukai lattice of $Ku(Y)$ is $K_{top}(Ku(Y)):=\{v\in K_{top}(Y): \chi([\o_Y(i)], v)=0, i=0,1,2 \}$ equipped with the Mukai pairing $(-,-):=-\chi(-,-)$. The (numerical) Mukai lattice $K_{num}(Ku(Y))$ is the image of the map $Ku(Y)\to \K$. For an object $E\in Ku(Y),$ refer to its class $[E]\in \k$ as its Mukai vector.
\end{dfn}

For any cubic fourfold $Y,$ the Mukai lattice $\k$ always contains two special classes: $$\lambda_1:= [i^*(\o_L(H))] \ \text{and} \ \lambda_2:=[i^*(\o_L(2H))],$$ where $L$ is a line in $Y.$ They generate a sublattice of $\k$ that is isomorphic to 
\begin{equation*}
    A_2 =
  \left( {\begin{array}{cc}
   2 & -1 \\
   -1 & 2 \\
  \end{array} } \right).
\end{equation*}

As shown in \cite[proposition 2.3]{AT}, if $Y$ is generic in the moduli of cubic fourfold, then $\k = A_2$. We call such a cubic fourfold very general.

\subsection{Stability conditions and Moduli}\label{sec2.2}
\begin{dfn}\cite{Bri}
A Bridgeland stability condition on $Ku(Y)$ is a pair $(Z, \mathcal A),$ where $Z: \k \rightarrow \mathbb C$ is a group homomorphism and $\mathcal A$ is a heart of a bounded t-structure of $Ku(Y)$ satisfying the following conditions:
\begin{enumerate}[(1)]
\item For any nonzero object $E \in \mathcal A,$ $Z(E):=Z([E])= r(E)e^{i\pi\phi(E)}, $ then $r(E)>0$ and $\phi(E)\in (0,1].$ ($\phi(E)$ is called the phase of $E$, and it defines a notion of semistability: an object $E\in \mathcal A$ is semistable if for any nonzero subobject $F\hookrightarrow E$, $\phi(F)\leq \phi(E).$)
\item For any object $E\in \mathcal A,$ E has a Harder-Narasimhan filtration $$0\hookrightarrow E_1 \hookrightarrow \ldots \hookrightarrow E_{n-1} \hookrightarrow E_n= E ,$$ such that quotient objects $A_i\cong E_i/E_{i-1}$ are semistable with decreasing phases, i.e. $\phi(A_i)>\phi(A_{i+1})$ for $i=1, 2, \ldots n.$

\item For a given norm $\| \cdot \|$ on $\k$, there exists a constant real number $C>0$ such that $$|Z(E)|<C\|[E]\|$$ for every semistable object in $E\in \mathcal A.$
\end{enumerate}
\end{dfn}  

\begin{dfn}
An object $E\in Ku(Y)$ is $\sigma$-semistable if $E[i]\in \mathcal A$ is semistable for some $i$.
\end{dfn}

The following two theorems are fundamental to us.

\begin{thm}\cite[theorem 1.2]{BLMS}
$Ku(Y)$ admits a Bridgeland stability condition $\sigma$.
\end{thm}

In the following, $\sigma$ will always be a stability condition on $Ku(Y)$ as constructed in \cite[proof of theorem 1.2]{BLMS}, and $\mathcal A$ always be the corresponding heart of t-structure. Let $M_\sigma(\lambda)$ be the moduli of $\sigma$-semistable objects in $\mathcal A$ with Mukai vector $\lambda \in \k.$ 

Let $L$ denote a line in $Y$, $C$ a generalized twisted cubic curve in $Y,$ and $I_L, I_C$ their ideal sheaves. $\mathbb R_{\o_Y}$ and $\mathbb L_{\o_Y}$ denote the right and left mutation with respect to the exceptional object $\o_Y.$ 
\begin{thm}[\cite{LLMS,LPZ}]\label{moduli} 

\begin{enumerate}[(a).]
    \item $M_\sigma(\lambda_1)$ is isomorphic to the variety of line $F(Y),$ parametrizing mutations of ideal sheaves of lines $F_L:= \mathbb L_{\o_Y}(I_L(1))[-1] = \text{ker}(H^0(I_L(1))\otimes \o_Y \twoheadrightarrow I_L(1)).$
    
    \item $M_\sigma(\lambda_1+\lambda_2)$ is isomorphic to $F(Y),$ parametrizing double mutations of ideal sheaves of lines $P_L:= \mathbb R_{\o_Y}(\mathbb L_{\o_Y}(I_L(1)))\otimes \o_Y(-1)[-1].$ In particular, it fits into a non-splitting extension $0\to \o_Y(-1)[1] \to P_L \to I_L \to 0.$
    
    \item Assume in addition that $Y$ does not contain any plane. $M_\sigma(2\lambda_1+\lambda_2)$ is isomorphic to the LLSvS eightfold $Z(Y).$
\end{enumerate}
\end{thm}

The following fact is shown in the proof of \cite[proposition 9.11]{BLMS}.
\begin{lemma}
Suppose that $F_L\in M_\sigma(\lambda_1)$ and $P_{L'}\in M_\sigma(\lambda_1+\lambda_2),$ then $\phi(F_L)<\phi(P_{L'})<\phi(F_L)+1.$
\end{lemma}

\subsection{Constant families of t-structures on $Ku(Y)$.}\label{sec2.3} We would like to consider $\f, \p$ as families of objects in the heart $\mathcal A \subset Ku(Y)$. The foundation for this is a combination of \cite{Kuz} and \cite{AP}, \cite{Po}, which we quickly review here.

First, we recall the following special case of a theorem of Kuznetsov:

\begin{dfn}\cite{Kuz}
Let $D^{[a, b]}(Y):=\{F\in D(Y):\mathcal H^i(F)=0, \text{ for any } i\notin [a, b]\}.$ Let $D$ be a triangulated subcategory of $D(X)$ and $\Phi: D\to D(Y)$ be a triangulated functor. We say $\Phi$ has finite amplitude if $\Phi(D\cap D^{[p, q]}(X))\subset D^{[p+a, q+b]}(Y)$ for some finite integers $a, b,$ for all $p, q\in \mathbb Z.$ 
\end{dfn}

\begin{thm}\cite[theorem 5.6]{Kuz}\label{sod}
Let $Y$ be a smooth projective variety with a semiorthogonal decomposition of its derived category $D^b(Y)=\langle D_1, \ldots, D_m \rangle,$ such that the projection functors to $D_i$ have finite amplitude. Let $S$ be a scheme of finite type over $\mathbb C$, write $Y_S:=Y\times S$ and $f:Y_S\to S,$ $p: Y_S\to Y$ for the projections. Then the derived category of $Y_S$ decomposes $$D^b(Y_S)=\langle D_{1,S}, \ldots, D_{m,S} \rangle.$$ In particular, if $ i: T\subset S$ is either an open immersion or a smooth point, then the functors $Li^*: D^b(Y_S)\to D^b(Y_T)$ and $Ri_*:D^b(Y_T)\to D^b(Y_S)$ respect the decompositions. 
\end{thm}

\begin{ex}
Let $Y$ be as above and suppose that $D^b(Y)=\langle Ku(Y), E_1, \ldots, E_m \rangle,$ where $\langle E_1, \ldots, E_m \rangle$ is an exceptional collection, and $Ku(Y):=\{F\in D^b(Y): R^\cdot Hom(E_i, F)=0, \text{ for all } i=1, \ldots, m\},$ then the projection functors are compositions of mutations. Since $Y$ is smooth and projective, mutations are of finite amplitude, and thus so are the projection functors. Thus \cref{sod} produces a triangulated subcategory $Ku(Y)_S$ of $D^b(Y_S),$ which we refer to as the family of Kuznetsov components over $S.$ 
\end{ex}

\begin{dfn}\cite{Po}
A t-structure $(D^{\leq 0}, D^{\geq 0})$ of a triangulated category $D$ is Noetherian if its heart $D^{\leq 0}\cap D^{\geq 0}$ is a Noetherian abelian category. It is close to Noetherian if there exist a Noetherian t-structure $(\hat D^{\leq 0}, \hat D^{\geq 0})$ of $D$ such that $\hat D^{\leq -1} \subset D^{\leq 0} \subset \hat D^{\leq 0}.$ 
\end{dfn}

\begin{ex}\cite[section 1.2]{Po}
Let $D$ be a triangulated category with $K_{num}(D)$ being finitely generated. Then given any Bridgeland stability condition $\sigma=(\mathcal A, Z)$ on $D$ with $Z:K_{num}(D)\otimes \mathbb C \to \mathbb C,$ the associated heart of t-structure $\mathcal A$ is close to Noetherian.
\end{ex}

\begin{dfn}\cite{AP}
A sheaf of hearts (or equivalently, t-structures) of $D^b(X)$ over a scheme $S$ is a functor $U\to \mathcal A_U,$ where $U\subset S$ is an open subset and $\mathcal A_U \subset D^b(X\times U)$ is the heart of a t-structure, such that the restriction functor $D^b(X\times U) \to D^b(X\times V)$ is t-exact, for any open immersion $V\subset U.$  
\end{dfn}

We would like to have the analog of the following result for $\mathcal A\subset Ku(Y):$

\begin{thm}\cite{AP, Po}\label{sheaft}
Suppose that $Y$ is a smooth projective variety and  $\mathcal A\subset D^b(Y)$ is a close to Noetherian and bounded heart. Let $S$ be a finite type scheme over $\mathbb C$. Then there is a sheaf of hearts $\mathcal A_S\subset D^b(Y_S),$ such that for any smooth point $i_s: s\hookrightarrow S,$ $Li_s^* \mathcal A_S \cong \mathcal A.$
\end{thm}

One way to obtain that is to use gluing of t-structures:
\begin{prop}\cite{BBD}\cite[lemma 3.1.1]{Po}  \label{gluet}
Suppose that $D$ is a triangulated category with a semiorthogonal decomposition $D=\langle D_1, \ldots, D_m \rangle,$ such that each inclusion functor $D_i \subset D$ admits left and right adjoints $pl_i, pr_i$, respectively. Then given a bounded t-structure $(D_i^{\leq 0}, D_i^{\geq 0})$ for each $D_i,$ there exists a bounded t-structure on $D$ given by $D^{\leq 0}:=\{ F\in D: pl_i(F)\in D^{\leq 0}_i, i=1, \ldots, m\}$ and $D^{\geq 0}:=\{ F\in D: pr_i(F)\in D^{\geq 0}_i, i=1, \ldots, m\}.$
\end{prop}

\begin{lemma}\label{ctn}
Notations and assumptions as in \cref{gluet}, in addition suppose that the functors $pr_i: D\to D_i$ have finite amplitude. If the t-structures $(D_i^{\leq 0}, D_i^{\geq 0})$ are all (close to) Noetherian, then there is a choice of shiftings for these t-structure, such that gluing the shifted t-structures give rise to a (close to) Noetherian one on $D$.
\end{lemma}
\begin{proof}
This is an application of \cite[lemma 3.1.2]{Po}. Since $pr_i$ are assumed to be of finite amplitude, one can shift the given t-structures so that $pr_i|_{ D_j}: D_j \to D_i$ is right t-exact with respect to the shifted t-structures, for every $j>i.$ Then by \cite[lemma 3.1.2]{Po}, we have $$D^{[a, b]}=\{F\in D: pr_i(F) \in D_i^{[a, b]}\}.$$

Now we first assume that the hearts $\mathcal A_i:= D_i^{\leq 0}\cap D_i^{\geq 0}$ are Noetherian. Let $\mathcal A$ be the heart of the t-structure that comes from gluing $(D_i^{\leq 0}, D_i^{\geq 0})$ with appropriate shifts as above. For any $F\in \mathcal A,$ we have a decomposition of $F$

{\centering
\begin{tikzcd}[column sep=1em]
  0  \arrow{rr} && F_{m} \arrow{rr} \arrow{dl}  && F_{m-1} \arrow{dl} & .... &  F_{1} \arrow{rr}  && F \arrow{dl}\\
& pr_m(F) \arrow[ul] && pr_{m-1}(F)\arrow[ul]  && .... && pr_1(F)\arrow[ul]
\end{tikzcd}. 
\par} 
\noindent Given an ascending chain of subobjects $E_0 \subset E_1 \subset E_2 \subset \ldots \subset F$ in $\mathcal A,$ then $ pr_i(E_\cdot)$ is an ascending chain of subojects of  in a Noetherian abelian category and therefore stable. We see that $E_0 \subset E_1 \subset E_2 \subset \ldots \subset F$ is stable. Hence, $\mathcal A$ is Noetherian.

Now suppose that $(D_i^{\leq 0}, D_i^{\geq 0})$ are close to Noetherian. By definition, we have a Noetherian t-structure $(\hat D_i^{\leq 0}, \hat D_i^{\geq 0})$ of $D_i,$ such that $\hat D_i^{\leq -1} \subset D_i^{\leq 0} \subset \hat D_i^{\leq 0}.$ Gluing $(\hat D_i^{\leq 0}, \hat D_i^{\geq 0})$, up to appropriate shiftings, gives a Noetherian t-structure $(\hat D^{\leq 0}, \hat D^{\geq 0})$ of $D,$ with $\hat D^{\leq -1} \subset D^{\leq 0} \subset \hat D^{\leq 0}.$ Thus, $(D^{\leq 0}, D^{\geq 0})$ is close to Noetherian.
\end{proof}

\begin{cor}
Let $Y$ be a smooth cubic fourfold and $\mathcal A\subset Ku(Y)$ be the heart of the t-structure associated to a Bridgeland stability condition on $Ku(Y).$ Let $S$ be a quasi-projective variety. Then there exists a sheaf of hearts $\mathcal A_S \subset Ku(Y)_S,$ such that for any smooth point $i_s: s\hookrightarrow S,$ $Li_s^*\mathcal A_S \cong \mathcal A.$
\end{cor}
\begin{proof}
Recall that $D^b(Y)=\langle Ku(Y), \mathcal O_Y, \mathcal O_Y(1), \mathcal O_Y(2)\rangle.$ Note that the triangulated subcategory $\langle \o_Y(i) \rangle \cong D^b(pt).$ Thus we can glue $\mathcal A$ with a choice of hearts of t-structures on $\langle \o_Y(i) \rangle$ to get a heart $ \mathcal C \subset D^b(Y)$ that is close to Noetherian, by \cref{ctn}. Then \cref{sheaft} produces a sheaf of t-structures $  \mathcal C_S$ of $D^b(Y_S).$ Consider $\mathcal A_S:=\mathcal C_S \cap Ku(Y)_S,$ it defines a bounded t-structure of $Ku(Y)_S.$ Moreover, since the semiorthogonal decompositions are compatible with base change as in \cref{sod}, we see that $\mathcal A_S$ is indeed a sheaf of hearts and $Li_s^*\mathcal A_S \cong \mathcal A$ for any smooth point $i_s: s\hookrightarrow S.$ 
\end{proof}

\begin{dfn}\cite[definition 3.3.1]{AP}
Let $S$ be a scheme of finite type over $\mathbb C.$ A family of objects in the heart $\mathcal A \subset Ku(Y)$ over $S$ is an object $F\in Ku(Y)_S,$ such that for every closed point $i_s: s\in S$ one has $Li_s^*F \in \mathcal A.$  
\end{dfn}

\begin{prop}\cite[corollary 3.3.3]{AP}
Let $S$ be a smooth quasi-projective variety, and $E$ be a family of objects in the heart $\mathcal A$ over $S,$ then $E \in \mathcal A_S.$
\end{prop}

\begin{ex}
Let $S$ be the moduli space $M_\sigma(\lambda_1)$ (or $M_\sigma(\lambda_1+\lambda_2)$), there exists a universal family $\mathcal F$ (resp. $\p$) over $Y_S.$ For any $s\in S,$ we have $\f_s \text{ (resp. } \p_s \text{)} \in \mathcal A \subset Ku(Y)$ in the heart associated to a Bridgeland stability condtion constructed in \cite{BLMS}, thus $\f$ (resp. $\p$) $\in \mathcal A_S.$ 
\end{ex}

\begin{dfn}
Let $S$ be a scheme of finite type, and $\mathcal A$ be the heart of a t-structure on $Ku(Y).$ Fix a family of objects $\e$ in the heart $\mathcal A$ over $S.$ A family of quotients in $\mathcal A$ of $\e$ over $S$ is a morphism $\e \to \p,$ whose restriction $\e_s \to \p_s$ to every closed point $s\in S$ is a surjection in $\mathcal A.$ This defines a relative quot functor \begin{equation}\label{Quot}
Quot_{\mathcal A, S}(\e, \lambda): (T\xrightarrow{g} S) \longrightarrow \left\{ 
  \begin{aligned}
  &\text{families of quotients of } \e_T \\ 
  &\text{in } \mathcal A  \text{ and of class } \lambda \text{ over } T
  \end{aligned}
\right\}.
\end{equation}
  
\end{dfn}

\section{Families of extensions} In this section, we review the description of families of extensions in \cite{lange}, then adapt it slightly for our study of the Voisin map in next section.

Throughout this section, let $X\to S$ be a projective and flat morphism between noetherian schemes, $\mathcal C\subset D^b(X)$ be a sheaf of heart over $S,$ $\f$ and $\p$ be two families of objects in $\mathcal C$. Given $\eta \in R^1Hom(\p, \f)$ that corresponds to an extension in $\mathcal C:$
\begin{equation}\label{ext}
    0\to \f \to
 \e \to\p \to 0,
\end{equation}
use  $\eta(s) \in R^1Hom_{X_s}(\p_s, \f_s)$ to denote the extension class represented by $$0\to \f_s \to \e_s \to \p_s \to 0,$$ the restriction of \cref{ext} to $s.$

\begin{dfn}\cite[definition 2.1]{lange}\label{foe}
\textit{A family of extensions of $\p$ by $\f$ over $S$} is a collection $\{\eta_s \in R^1Hom_{X_s}(\p_s, \f_s)\}_{s\in S}$ such that there exists an open cover $(U_i)_{i\in I}$ of $S$ with $\eta_i\in R^1Hom_{f^{-1}(U_i)}(\p _{f^{-1}(U_i)}, \f _{f^{-1}(U_i)})$ for each $i\in I,$ satisfying $\eta_i(s) = \eta_s$ for all $s \in U_i.$ 
\end{dfn}

As in \cite{lange}, we will be considering two functors from the category of noetherian schemes over $S$ 
to the category of sets:
\begin{equation}\label{fun1}
(g: T\rightarrow S) \longrightarrow \left\{ 
  \begin{aligned}
  &\text{families of extensions } \\ 
  &\text{of}\ \p_T\ \text{by}\ \f_T\ \text{over}\ T
  \end{aligned}
\right\},
\end{equation}
and 
\begin{equation}\label{fun2}
(g: T\rightarrow S) \longrightarrow \left\{ 
  \begin{aligned}
  &\text{families of nonsplitting exten- } \\ 
  &\text{sions of}\ \p_T\ \text{by}\ \f_T\ \text{over}\ T
  \end{aligned}
\right\}\Big{/}H^0(T,\mathcal O_T^*).
\end{equation}
\noindent In particular, the representability of the second functor will be crucial for us to resolve the Voisin map via extensions. The following definition is the key to this question:

\begin{dfn}\cite[section 1]{lange}
 Define $$R\hom_f(\p, \f):= Rf_*R\hom(\p,\f)\ , \ \ext^i(\p,\f):= \mathcal H^i(R\hom_f(\p,\f)).$$ 
\end{dfn}

\begin{rmk}\label{1strmk}
Suppose in addition that $S$ is affine, then $\ext^i(\p, \f) \cong \widetilde{R^iHom(\p, \f)}.$ Therefore, 
$\ext^i(\p, \f)$ is the sheaf on $S$ associated to the presheaf $$U \to R^iHom_{f^{-1}(U)}(\p _{f^{-1}(U)}, \f _{f^{-1}(U)}).$$

    
    
    

\end{rmk}

\begin{lemma}\cite[corollary 1.2]{lange}\label{keylemma}
\begin{enumerate}[(a).]
    \item $R^\dot \hom_f(\p, \f)$ is quasi-isomorphic to a locally free complex $\mathcal W_\dot.$
    \item Given a Cartesian diagram, where $T$ is a Noetherian scheme, {\center
\begin{tikzcd}[column sep=tiny]
 X_T \arrow[rr,"g"] \arrow[d, "f_T"] & & X  \arrow[d, "f"]\\
T \arrow[rr, "g"] & & S,
      \end{tikzcd}\par} and any sheaf $\mathcal G$ on $T,$  we have $\mathcal H^i(\mathcal W_\dot \otimes \mathcal G)\cong \mathcal Ext^i_{f_T}(Lg^*\p,\f\otimes^L \mathcal G).$ 
\end{enumerate}
\end{lemma}

\begin{proof}
\begin{enumerate}[(a).]
    \item Let $\mathcal I_\dot$ be an injective replacement of $\f$ that has finitely many non-zero terms, $L$ a $f$-ample line bundle on $X.$ Then for sufficient large $n,$ we have $R\hom_f(L^{-n}, I_k)=f_*\hom(L^{-n},I_k)$ locally free for all $k$. Now take a locally free replacement of $\p$ with each term being sufficiently negative in the above sense, we obtain a double complex whose entries are all locally free sheaves. $R\hom_f(\p, \f)$ is represented by its total complex $\mathcal W_\dot$. 
    
    \item $\mathcal H^i(\mathcal W_\dot \otimes \mathcal G)\cong \mathcal H^i(Lg^*Rf_*R\hom(\p, \f)\otimes^L \mathcal G) \cong \mathcal Ext^i_{f_T}(Lg^*\p,\f\otimes^L \mathcal G).$ For the last isomorphism, we use a base change theorem; as we have $T$ being Noetherian, $f$ flat and projective, and $R\hom_f(\p,\f)$ quasi-isomorphic to a locally free complex of finite length.  
\end{enumerate}
\end{proof}


\noindent By \cref{keylemma}, we have a natural base change morphism $ \ext^i(\p, \f)\otimes k(s) \to R^iHom(\p_s, \f_s)$ for each $s\in S,$ where $k(s)$ is the residue field. 

\begin{dfn}\label{cwbc}
 We say $\ext^i(\p, \f)$ commutes with base change if the maps $\ext^i(\p, \f)\otimes k(s) \to R^iHom(\p_s, \f_s)$ are isomorphisms for all points $s\in S$ 
\end{dfn}

\begin{prop}\cite[proposition 2.3]{lange}\label{lange2.3}
Let $f:X\to S,$ $\p, \f $ be as before. Suppose that $\ext^1(\p, \f)$ commutes with base change, then there is a canonical bijection between the set of families of extensions of $\p$ by $\f$ over $S$ and $H^0(S, \ext^1(\p,\f)).$
\end{prop}

 However, the condition that $\ext^1(\p, \f)$ "commutes with base change" will not hold in our case. We observe the following:

\begin{lemma}\label{psi1} Suppose that $R^{i}\hom_f(\p,\f)=0$ for all $i\leq 0.$ Then for any morphism $g:T\to S$, where $T$ is a reduced Noetherian scheme, there is a canonical bijection between the set of families of extensions of $\p_T$ by $\f_T$ over $T$ and $Hom(\mathcal O_T, Lg^*R^\dot \hom_f(\p,\f)[1])$.
\end{lemma}


\begin{proof} 
Given $\phi \in Hom(\mathcal O_T, Lg^*R^\dot \hom_f(\p,\f)[1]),$ and a point $t\in T$ , we get $\phi_t:=\phi\otimes^Lk(t): k(t)\to Lg^*R^\dot \hom_f(\p,\f)\otimes k(t)[1].$ Note that by \cref{keylemma}, $Lg^*R^\dot \hom_f(\p,\f) \otimes k(t)[1] \cong \bigoplus_i R^iHom_{X_t}(\p_t, \f_t)[1-i],$ thus we obtain a collection $\{\phi_t\in R^1Hom_{X_t}(\p_t, \f_t)\}_{t\in T}.$ 
 
 Moreover, if we choose an affine open cover $\{U_i\}$ of $T,$ then by \cref{1strmk} (c), $\phi_i:=\phi _{U_i} \in Hom(\mathcal O_{U_i}, Lg^*R^\dot \hom_f(\p,\f) _{U_i}[1])= R^1Hom_{X_i}(\p _{i}, \f _{i}),$ where $X_i$ (resp. $\f_i, \p_i$) is the base change of $X$ (resp. $\f, \p$) to $U_i.$ Suppose that $\phi_i$ corresponds to an extension $0\to \f_i \to \e_i \to \p_i \to 0,$ and for any $t\in U_i,$ the class $\phi_{i}(t)$ represents the restriction $0\to \f_t \to \e_t \to \p_t \to 0.$ Then $\phi_i(t)$ is the image of $$k(t) \xrightarrow{\phi_i \otimes^L k(t)} \mathcal Ext^1_{f_i}(\p_i, \f_i)\otimes k(t) \to R^1Hom_{X_t}(\p_t,\f_t),$$  where $f_i$ is the base change of $f$ to $U_i.$ Note that $\phi_i(t) = \phi_t.$ Hence, $\{\phi_t\in R^1Hom_{X_t}(\p_t, \f_t)\}_{t\in T}$ is a family of extensions. 
 
 Conversely, given a family of extensions over $T,$ we have $\{U_i, \eta_i\}$ as in \cref{foe}. Without loss of generality, we may assume all $U_i$ are affine. Again by \cref{1strmk} and \cref{keylemma}, $\widetilde{Ext^1_{X_{i}}(\p _{U_i},\f _{U_i})}\cong \mathcal Ext^1_{f_i}(\p _{U_i},\f _{U_i})\cong \mathcal H^0(Lg^*R^\dot \hom_f(\p,\f) _{U_i}[1]).$ Thus $\eta_i \in Hom(\mathcal O_{U_i}, \mathcal H^0(Lg^*R^\dot \hom_f(\p,\f) _{U_i}[1])) \cong Hom(\mathcal O_{U_i}, Lg^*R^\dot \hom_f(\p,\f) _{U_i}[1]).$ By definition of a family of extensions, we have $\eta_i\otimes^L k(t) = \eta_j \otimes^L k(t),$ for all $t\in U_{ij}.$ 
 
Now the assumption that $R^i\hom_f(\p,\f)=0, i \leq 0$ means $R \hom(\p, \f)[1]\in D^b(S)^{\geq 0}.$ Then according to \cite[corollary 2.1.22]{BBD} (see also \cite[proposition 2.1.10]{Lie}), $$U \to Hom(\mathcal O_U, Lg^*R^\dot \hom_f(\p,\f)_{U}[1])$$ is a sheaf on $T$. Thus the classes $\eta_i$ glue, giving an element $\eta \in Hom(\mathcal O_{T}, Lg^*R^\dot \hom_f(\p,\f)[1]).$
 

\end{proof}

For later use, we recall a well-known theorem:

\begin{thm}\cite[theorem A.5 (i)]{LK}\label{bcthm}
Suppose that the base change map $\ext^i(\p, \f)\otimes k(s) \to R^iHom_{X_s}(\p_s, \f_s)$ is surjective for some $s\in S,$ then there is an open neighborhood $U$ of $s$ such that $\ext^i(\p, \f)\otimes k(s') \to R^iHom_{X_s'}(\p_s', \f_s')$ is isomorphic for all $s'\in U.$ 
\end{thm}

\begin{rmk}
The assumption that  $R^i\hom_f(\p,\f)=0, i \leq 0$ in \cref{psi1} can be obtained, for example, by taking $\p, \f$ to be families of stable objects with slopes $\phi(\p_s) > \phi(\f_s).$ Because in that case we have $R^iHom(\p_s, \f_s)=0$ for all $i\leq 0$ and all $s\in S,$ then by \cref{bcthm} the claim follows. 
\end{rmk}

Next, we characterize the set of families of extensions whose restriction to any closed point does not split.

\begin{lemma}\label{psi2}
With the same assumption and notations as in \cref{psi1}, let $\mathcal W_\dot$ denote $R^\dot \hom_f(\p,\f)[1],$ then there is a canonical bijection: 
\begin{equation}
\left\{ \begin{aligned} 
&\psi: Lg^*(\mathcal W_\dot^\vee) \rightarrow \mathcal L_T, \text{where}\ \mathcal L_T\in Pic(T)\\
& \ \text{such that}\ \mathcal H^0(\psi): L^0g^*(\mathcal W_\dot^\vee) \twoheadrightarrow \mathcal L_T 
\end{aligned}
\right\} \longleftrightarrow \left\{ 
  \begin{aligned}
  &\text{families of nonsplitting exten- } \\ 
  &\text{sions of}\ \p_T\ \text{by}\ \f_T\ \text{over}\ T
  \end{aligned}
\right\}\Big{/}H^0(T,\mathcal O_T^*).
\end{equation}
\end{lemma}

\begin{proof}
By \cref{psi1}, the right hand side of the bijection above is the same as $$\{\phi: \mathcal L_T \rightarrow Lg^*(\mathcal W_\dot): \phi\otimes^L k(t)\neq 0, \forall t\in T\}.$$ We want to show that this is equivalent to the left hand side. By \cref{dual}, we have $\psi:=\phi^\vee: Lg^*(\mathcal W_\dot^\vee) \to \mathcal L^\vee_T$. Applying $-\otimes^L k(t)$ to the triangle $\mathcal L_T\xrightarrow{\phi} Lg^*(\mathcal W_\dot)[1] \to cone(\phi),$ we see that $\phi\otimes^L k(t)\neq 0$ if and only if $\mathcal H^{-1}(cone(\phi)\otimes^L k(t))\xleftarrow{\sim}\mathcal H^{-1}(Lg^*(\mathcal W_\dot)\otimes^L k(t))\cong 0,$ because $\mathcal W_\dot \in D^b(S)^{\geq 0}$ as we have seen in \cref{psi1}. On the other hand, we have $ Lg^*(\mathcal W_\dot^\vee)\xrightarrow{\psi} \mathcal L_T^\vee \to cone(\psi),$ and $\mathcal H^0(\psi)$ is surjective if and only if $\mathcal H^0(cone(\psi)\otimes^L k(t))=0.$ Note that $cone(\phi)^\vee [1] \cong cone(\psi),$ so $\mathcal H^{-1}(cone(\phi)\otimes^L k(t))\cong \mathcal H^0(cone(\psi)^\vee \otimes^L k(t)) \cong \mathcal H^0(cone(\psi)\otimes^L k(t))^\vee.$ Thus, the conditions on $\phi$ and $\psi$ are equivalent and the bijection is given by dualizing.
\end{proof}
 Therefore, we can rewrite the functor \eqref{fun2} as follow:
\begin{dfn}\label{psi}
The functor of families of non-splitting extensions of $\p$ by $\f$ is
\begin{equation} 
\Psi: (g: T\rightarrow S) \longrightarrow \left\{ \begin{aligned} 
&\psi: Lg^*(\mathcal W_\dot^\vee) \rightarrow \mathcal L_T, \text{where}\ \mathcal L_T\in Pic(T)\\
& \ \text{such that}\ \mathcal H^0(\psi): L^0g^*(\mathcal W_\dot^\vee) \twoheadrightarrow \mathcal L_T 
\end{aligned}
\right\},
\end{equation}
from the category of Noetherian schemes over $S$ to the category of sets, where $\mathcal W_\dot:= R^\dot \hom_f(\p,\f)[1]$.
\end{dfn}

\begin{rmk}
Note that we do not restrict the functor to the category of reduced Noetherian schemes, even though $\cref{psi1}$ requires reducedness of $T.$ Indeed, we can take $Hom(\o_T,  Lg^*R^\dot \hom_f(\p,\f)[1])$ as our definition of families of extension over $T.$ 
\end{rmk}

\section{The Voisin map}\label{sec4}

Now we specialize to our case: suppose that $Y$ is a cubic fourfold not containing any plane, $Ku(Y)$ is the Kuznetsov component of $Y$, $M_1:=M_\sigma(\lambda_1)$ and $M_2:=M_\sigma(\lambda_1+\lambda_2)$ are the moduli spaces, then we set $S:= M_1\times M_2,$ $X:= M_1\times M_2 \times Y$, $f$ to be the projection, and $\f$ (resp. $\p$) to be the pullback of a universal families on $M_1 \times Y$ (resp. $M_2 \times Y$). Recall that $\f, \p$ are families of objects in a heart $\mathcal A\subset Ku(Y)$ associated to the stability condition $\sigma$ constructed in \cite[theorem 1.2]{BLMS}. Also we specialize the functor $\Psi$ (\cref{psi}) to this setting. We may use the notation $Ext^i(-,-):=R^iHom(-,-).$

Via a canonical identification $\M \cong F(Y)\times F(Y),$ we define the incident locus $\Gamma:=\{(L_1, L_2)\in \M : L_1 \cap L_2 \neq \emptyset \}$ and the type II locus $\Delta_2:=\{L\in \Delta\cong F(Y): \mathcal N_{L/Y}\cong \mathcal O(1)^2\oplus \mathcal O\}$ of the diagonal $\Delta \subset \M$ (see e.g. \cite{deb}).  

\begin{lemma}\label{extfp} Recall $F_{L}$ and $P_{L}$ from \cref{moduli}, then $Ext^i(F_{L_1}, P_{L_2})=0$ unless $i=0\ or \ 1$,  and
\begin{equation}
Ext^1(F_{L_1}, P_{L_2}) \cong \left\{ 
  \begin{aligned}
  & \mathbb C, \ \ \ (L_1, L_2)\in \M\setminus\Gamma,  \\ 
  & \mathbb C^2, \ \ (L_1, L_2) \in \Gamma\setminus\Delta_2, \\
  & \mathbb C^3, \ \ (L_1, L_2) \in \Delta_2.
  \end{aligned} \right.
\end{equation}
\end{lemma}
\begin{proof}
First, recall the defining sequences of $F_L$ and $P_L$ respectively: $$F_L\to H^0(I_{L/Y}(H))\otimes \mathcal O_Y \to I_{L/Y}(H),\ \ \ \mathcal O_Y(-H)[1]\to P_{L} \to I_{L/Y}.$$ 
One can easily verify that $Hom^{-1}(F_{L_1}, \mathcal O_Y(-H)[1])=0,$ and thus by Serre duality on $Ku(Y)$ (since both $F_{L}, P_{L} \in Ku(Y)$ and $Ku(Y)$ is a full subcategory of $D^b(Y)$), we get $Ext^i(F_{L_1}, P_{L_2})=0$ unless $i=0, 1\ or\ 2.$ Also, according to \cite{LLMS}, $F_{L_1}$ and $P_{L_2}$ are stable with respective to a stability condition with phases $\phi(F_{L_1})<\phi(P_{L_2}),$ therefore $Ext^2(F_{L_1}, P_{L_2})=Hom(P_{L_2}, F_{L_1})=0.$ 

Next, recall that $\chi(F_{L_1}, P_{L_2})=-(\lambda_1, \lambda_1+\lambda_2)=-1 = hom(F_{L_1}, P_{L_2})-ext^1(F_{L_1}, P_{L_2})$, so we may just compute $Hom(F_{L_1}, P_{L_2}).$ Using long exact sequences derived from the two defining sequences, we obtain: 
\begin{equation*}
Hom(F_{L_1}, P_{L_2}) \cong Hom(I_{L_1/Y}(H), \mathcal O_{L_2}) \cong \left\{ 
  \begin{aligned}
  &  0, \ \ \ (L_1, L_2)\in \M\setminus\Gamma,  \\ 
  & \mathbb C^1, \ \ (L_1, L_2) \in \Gamma\setminus\Delta_2, \\
  & \mathbb C^2, \ \ (L_1, L_2) \in \Delta_2.
  \end{aligned} \right.
\end{equation*}
\end{proof}

\begin{cor}\label{exacts} We have the following two exact sequences:
\begin{enumerate}[(a).]
    \item $0\to \mathcal V_0 \xrightarrow{\alpha} \mathcal V_1 \to \ext^1(\f, \p) \to 0,$ where $\mathcal V_0, \mathcal V_1$ are locally free sheaves and $\ext^1(\f,\p)$ commutes with base change;
    \item $0\to \ext^1(\p, \f)\to \mathcal W_1 \xrightarrow{\beta} \mathcal W_2 \to \ext^2(\p, \f) \to 0,$ where $\mathcal W_1, \mathcal W_2$ are locally free sheaves, $\ext^2(\p,\f)$ commutes with base change and $\ext^1(\p,\f)$ is a line bundle. 
\end{enumerate}
\end{cor}

\begin{proof}

\begin{enumerate}[(a).]
    \item By \cref{keylemma}, we have a locally free replacement $\mathcal{V}_\dot$ of $R^\dot\hom_f(\f, \p).$ By \cref{bcthm} and \cref{extfp}, $\mathcal V_\dot$ at most have cohomologies at degree 0 or 1. Thus we have $$0\to \hom_f(\f, \p) \to \mathcal V_0\to \mathcal V_1 \to \ext^1(\f, \p)\to 0.$$  $\hom_f(\f,\p)$ must be zero as well, because otherwise it would be a torsion sheaf again by \cref{bcthm} and \cref{extfp}, but $\mathcal V_0$ is locally free.
    
        Since $f: \M \times Y \to \M$ and $\p$ are both flat over $\M,$ given any locally free sheaf $\mathcal L,$ $\hom_f(\mathcal L, \p)$ is flat, thus for point $s\in \M$ we have a Grothendieck spectral sequence associated to $R\hom_f(-,\p)\circ -\otimes^L k(s),$ whose $E_2$ page concentrates in two rows and in particular gives $$\ext^1(\f, \p)\otimes k(s) = E_2^{0,1} \cong E_{\infty}^1= Ext^1_{\footnotesize{Ku(Y)}}(\f_s,\p_s).$$
    \item Similarly, we have a locally free replacement $\mathcal{W}_\dot$ of $R^\dot\hom_f(\p, \f)$ and moreover an exact sequence $$0\to \ext^1(\p,\f) \to \mathcal W_1 \to \mathcal W_2 \to \ext^2(\p, \f) \to 0.$$ The Grothendieck spectral sequence this time yields $$\ext^2(\p,\f)\otimes k(s) \cong Ext^2_{\footnotesize}(\p_s,\f_s).$$ Thus, $\ext^2(\p,\f)$ commutes with base change. Note that $\ext^1(\p,\f)$ is of rank one and reflexive, and thus a line bundle.

\end{enumerate}
\end{proof}


\begin{rmk}
On the open subset $U:=\M \setminus \Gamma,$ $\ext^1(\p, \f)_U$ commutes with base change since $\ext^2(\p,\f)=0$, then by \cite[corollary 4.5]{lange} (which in addition needs the fact that $\hom_f(\p,\f)=0$) we have a universal nonsplitting extension $$0\to \f _U \boxtimes (\ext^1(\p, \f) _U)^\vee \to \mathcal E _U \to \p _U\to 0.$$ By results of \cite{LLMS} and \cite{LPZ}, $\mathcal E_s$ is stable with respect to a Bridgeland stability condition $\sigma$ for all closed points $s\in U.$ As $\M$ is reduced, this $\e$ defines a rational map $v: \M \dashrightarrow M_\sigma(2\lambda_1+\lambda_2).$ We refer to this as the Voisin map. 
\end{rmk}

\begin{lemma}\label{dual}
${(\mathcal V_0 \xrightarrow{\alpha} \mathcal V_1)}^\vee [-2] \cong (\mathcal W_1 \xrightarrow{\beta} \mathcal W_2).$ 
\end{lemma}
\begin{proof}
Recall that $(\mathcal V_0 \xrightarrow{\alpha} \mathcal V_1)\cong R f_*R\mathcal Hom(\f, \p)$ and $\mathcal W_1 \xrightarrow{\beta} \mathcal W_2 \cong R f_*R\mathcal Hom(\p, \f).$ By Grothendieck-Verdier duality,
\begin{align*} 
R\hom(Rf_*R\mathcal Hom(\f, \p), \mathcal O_\M) 
              & \cong Rf_*R\hom(R\hom(\f, \p), \pi^*\omega_Y[4])\\
             & \cong Rf_*R\hom(\p, S_Y(\f)) \\
              &\cong Rf_*R\hom(\p, i^*\circ S_Y(\f)) \\
              &\cong Rf_*R\hom(\p, \f)[2],
\end{align*}
where $\pi: \M \times Y \to Y$ is the projection, $S_Y$ (resp. $S_{Ku}$)denotes the Serre functor on $\M \times Y$ (resp. $Ku(\M \times Y)$), and $i^!$ (resp. $i^*$) is the right (resp. left) adjoint to the inclusion $Ku(\M\times Y)\hookrightarrow D^b(\M \times Y).$ For the last isomorphism, we use that $i^*\circ S_Y \cong S_{Ku}\circ i^!.$ 
\end{proof}

\begin{prop}
 $R\hom_f(\f, \p)[1]\cong \ext^1(\f,\p)\cong I_\Gamma,$ where $I_\Gamma$ is the ideal sheaf of the reduced scheme structure on $\Gamma.$ 
\end{prop}
\begin{proof}
 Since $\ext^2(\p, \f)$ commutes with base change, it is a torsion sheaf supported on $\Gamma,$ by the computation in \cref{keylemma}. Meanwhile $\ext^2(\p, \f)$ is the cokernel of $\beta: \mathcal W_1 \rightarrow \mathcal W_2$, thus the degeneracy locus of $\beta$ is exactly $\Gamma.$ Note that we have the  Eagon-Northcott resolution $0\to \mathcal W_2^\vee \xrightarrow{\beta^\vee} \mathcal W_1^\vee \to I_\Gamma \to 0$ of the ideal sheaf of the reduced scheme structure on $\Gamma$ (see e.g. \cite[theorem 2.11 \& 2.16]{BV}). By \cref{dual}, this resolution is exactly $0\to \mathcal V_0 \xrightarrow{\alpha} \mathcal V_1 \to \ext^1(\f, \p) \to 0.$
\end{proof}

\begin{thm}\label{mainthm}
Let $b: Bl_{\Gamma}(\M) \to \M$ be the blowup of $\M$ along $I_\Gamma.$ 
\begin{enumerate}[(a).]
\item
The functor $\Psi$ is represented by $Bl_{\Gamma}(\M).$ Consequently, the Voisin map can be resolved by $Bl_{\Gamma}(\M)$:
{\center
\begin{tikzcd}[column sep=tiny]
 & Bl_{\Gamma}(\M)  \arrow[ld, "b"]{dashed} \arrow[rd, "q"]\\
\M \arrow[rr, dashed, "v"]& & M_{\sigma}(2\lambda_1+\lambda_2)
      
\end{tikzcd}\par}

\item If $K_{num}(Ku(Y)) \cong A_2,$ then $q: Bl_{\Gamma}(\M)\to M_{\sigma}(2\lambda_1+\lambda_2)$ is a relative Quot scheme of objects in $\mathcal A_\sigma,$ with a universal quotient $\e \to \p.$

\end{enumerate}

\end{thm}

\begin{proof}

\begin{enumerate}[(a).]
\item By \cref{psi2}, the functor $\Psi$ can be interpreted as \begin{equation}
\Psi(g: T\rightarrow S) = \left\{ \begin{aligned} 
&\psi: Lg^*(I_\Gamma) \rightarrow \mathcal L_T, \text{where}\ \mathcal L_T\in \\
& Pic(T),\ \text{s.t.}\ \mathcal H^0(\psi): g^*(I_\Gamma) \twoheadrightarrow \mathcal L_T 
\end{aligned}\right\}.
\end{equation}
\noindent We want to show that this is represented by $b:\bl\to \M$.

First recall that on $\bl$ there is a universal quotient $b^*I_\Gamma \twoheadrightarrow \mathcal O(-E),$ where $E$ is the exceptional divisor. Thus, given any $g:T\to \M$ that factors through $T\xrightarrow{\bar{g}} \bl \xrightarrow{b} \M,$ we have $g^*I_\Gamma \twoheadrightarrow \bar g^*\mathcal O(-E).$

Conversely, given $Lg^*I_\Gamma \to g^*I_\Gamma \twoheadrightarrow \mathcal L_T,$ since $Lg^*I_\Gamma\cong (g^*\mathcal V_0 \to g^*\mathcal V_1)[1],$ we have $g^*\mathcal V_1\twoheadrightarrow g^*I_\Gamma \twoheadrightarrow \mathcal L_T,$ which then induces an embedding $\iota_1: T\hookrightarrow \mathbb P_{\M}(\mathcal V_1).$ Also we have $\iota_2: \bl \hookrightarrow \mathbb P_{\M}(\mathcal V_1)$ induced by $\mathcal V_1\twoheadrightarrow I_\Gamma.$ For any $t\in T,$ because $\mathcal V_1\otimes k(t)\twoheadrightarrow I_\Gamma \otimes k(t) \twoheadrightarrow \mathcal L_T\otimes k(t),$ $\iota_1(t)$ lies in the image of $\iota_2,$ and therefore $\iota_1$ factors through a morphism $\iota_3: T\to \bl.$ 

Thus, $\bl$ represents the functor $\Psi,$ admitting a universal quotient $b^*I_\Gamma \twoheadrightarrow \mathcal O(-E).$ By \cref{psi1}, this corresponds to a universal family of nonsplitting extensions of $\p$ by $\mathcal O(-E)\boxtimes \f$. In particular, there is a open covering of $\bl$ with extension classes $\{U_i, \eta_i\}$ represents this family. On each $U_i$ we have an extension $$0\to \mathcal O_{U_i}(-E)\boxtimes \f_i \to \e_i \to \p_i\to 0.$$ For any closed point $s\in U_i,$ $(\e_i)_s$ is a $\sigma$-stable object in $Ku(Y)$ of class $2\lambda_1+\lambda_2,$ according to results in \cite{LLMS}. Moreover, $\e_i$ is flat over $U_i.$ Since $\bl$ is reduced and so is $U_i$, $\e_i$ defines a morphism $ q_i: U_i \to M_\sigma(2\lambda_1+\lambda_2).$ As $(\e_i)_s\cong (\e_j)_s$ for any $s\in U_{ij},$ these morphisms glue, giving a morphism $q: \bl \to \m$ that resolves the Voisin map. 

\item Suppose that we have a morphism $\phi: B \to Z(Y),$ with a family of quotients $\e_B \to \p_B$ over $B$, where $\e_B$ is a family of objects in $\m$ defining the map $\phi,$ $\p_B$ is a family of obejcts of class $\lambda_1+\lambda_2,$ and both are flat with respect to $\mathcal A.$ If $Y$ is generic, then according to Mukai's lemma and \cite[proposition 6.4 \&  6.5]{LLMS}, fibers of $\p_B$ are in fact stable. Thus we have a family of nonsplitting extensions over $B$ and therefore a morphism $B\to Bl_\Gamma$ over $Z(Y).$ The only thing left to show is the existence of a universal quotient on $Bl_\Gamma.$

Note that the quotients $\e_i\to \p_i$ over $U_i$ in part (a) does not necessarily glue; it is a twisted sheaf a priori. However, we use elementary modification (see e.g. \cite{AB}) to show that in this case it does in fact come from a global quotient over $Z(Y).$ 

Because $I_\Gamma \cong R^\dot \hom_f(\f, \p)[-1] \cong R^\dot \hom_f(\p,\f)^\vee [1],$ then by \cref{psi1} the inclusion $I_\Gamma \hookrightarrow \mathcal O_{\M}$ corresponds to a family of extensions of $\p$ by $\f$, whose restriction to any point in $\Gamma$ splits. Using the Grothendieck spectral sequence $H^i(\ext^j(\p,\f)) \Rightarrow Ext^{i+j}(\p, \f)$ and $\hom_f(\p, \f)=0,$ we see that $Ext^1(\p,\f)\cong H^0(\ext^1(\p,\f))\cong H^0(\mathcal H^0(I_\Gamma^\vee))\cong \mathbb C,$ thus the family of extensions above is in fact global. 

Similarly, the pullback $b^*I_\Gamma \to \mathcal O_{Bl}$ corresponds to a global family of extensions 
\begin{equation}\label{degenext}
  0\to b^*\f \to \mathcal E' \to b^*\p \to 0  
\end{equation}
on the blowup $Bl:= \bl,$ such that $\mathcal E' _E = b^*\p _E \oplus b^*\f _E,$ where $E$ is the exceptional divisor and we abuse notation $b: Bl \times Y \to \M \times Y$. Consider the short exact sequence in the heart $\mathcal A_\sigma(Bl)$ $$0\to \e \to \mathcal E' \to b^*\f _E \to 0,$$ where $\e$ is by definition the kernel of the composition $\mathcal E'\to \mathcal E' _E \to b^*\f _E.$ Restrict $\e$ to $E$ and use the octahedral axiom of triangulated categories, we obtain another short exact sequence $$0 \to b^*\f(-E) _E \to \e  _E \to b^*\p _E \to 0,$$ which is non-splitting because by construction of $\e$ we have $Hom(\e  _E, b^*\f _E)=0.$ This also implies that the composition $\e \to \mathcal E' \to b^*\p$ is surjective, with kernel $b^*\f(-E),$ i.e. we have another global extension
\begin{equation}\label{univext}
\begin{tikzcd}[column sep=tiny]
0 \arrow[r]  & b^*\f(-E)  \arrow[d] \arrow[r] & \e \arrow[r] \arrow[d] & b^*\p \arrow[r] \arrow[d, equal] & 0\\
0 \arrow[r]  & b^*\f  \arrow[r] & \e' \arrow[r] & b^*\p \arrow[r]  & 0      
\end{tikzcd}
\end{equation}

Recall that the lower row in the above diagram corresponds to $b^*I_\Gamma \to \mathcal O_{Bl},$ considered as an element in $Hom(\o, R^\dot \hom_f(b^*\p, b^*\f)).$ Therefore, the upper row corresponds to $b^*I_\Gamma \twoheadrightarrow \o_{Bl}(-E),$ considered as an element in  $Hom(\o, R^\dot \hom_f(b^*\p, b^*\f(-E))),$ which get mapped to $b^*I_\Gamma \to \mathcal O_{Bl}$ under the map induced by natural inclusion $\o(-E)\hookrightarrow \o.$ By part (a), the upper row is a global family of non-splitting extensions and universal. In particular, we obtain a universal quotient $\e \twoheadrightarrow b^*\p.$ Thus $\bl$ is a relative Quot scheme over $\m.$

\end{enumerate}
\end{proof}

\begin{rmk}
The map $q: \bl \to \m$ is surjective and of degree six. These follow from Voisin's argument \cite{Voi}. Recall that $\m$ and $ M_\sigma(\lambda_1)$ generically parametrize $F_C:=L_{\mathcal O_Y}(I_{C/S}(2H))[-1]$ and $F_l:=L_{\mathcal O_Y}(I_{l/Y}(H))[-1]$ respectively. Thus, $Hom(F_L, F_C) = Hom(I_{l/Y}, I_{C/S}(H)).$ The latter ideal sheaf is indeed $\mathcal O_S(l_1-l_2)$, where $l_1, l_2$ are two skew lines in $S$, provided that $F_C$ generic. Therefore, we have a nontrivial morphism from $F_{l_1}$ to $F_C.$ By \cite{LLMS}, $F_C$ is a nontrivial extension of some $P_{l'}$ by $F_{l_1}.$ This shows the Voisin map is dominant and thus $q$ is surjective. Moreover, the degree is six, as also shown in \cite[proposition 4.8]{Voi}.
\end{rmk}

\begin{rmk}
We record here a computation suggesting that part (b) of \cref{mainthm}, namely, the Voisin map being resolved by a degree six relative Quot scheme, may still be true without the assumption that $Y$ is very general.

Suppose that $Y\in \mathcal C_d$ with $d=2r(r-1)+2,$ where $\mathcal C_d$ is the Hassett divisor parametrizing Hodge special cubic fourfolds of discriminant $d$, see \cite{Ha}. In particular, $Y$ is not very general. If $Y$ is moreover general within $\mathcal C_d,$ then according to \cite{AT}, $Ku(Y)\cong D^b(S)$ for some K3 surface $S$ admitting a polarization $L$ with $L^2=d.$ Consider two Mukai vectors of $S:$ $$(1, 0, -1)\ ,\ (r, -L, r-1).$$ Note that these two Mukai vectors also generate a sublattice that is isomorphic to $A_2.$ By a lattice-theoretic result \cite[Theorem 2.4]{LP} and the derived Torelli theorem for K3, there exists an auto-equivalent $\Phi$ of $D^b(S),$ such that $\Phi(\lambda_1)=(r, -L, r-1)$ and $\Phi(\lambda_1+\lambda_2)=(1, 0, -1).$

On the other hand, we can use a method by D. Johnson \cite{Joh} to compute the length of the finite Quot scheme $Quot(V, (1, 0, -1)),$ where $V$ is a general stable sheaf of class $(r+1, -L, r-2)$ and the quotients are in the category of coherent sheaves of $S$ (which can never be the heart of any Bridgeland stability condition though). Roughly speakingly, the number of surjections to ideal sheaves of two points is given by a top Chern class $c_4(V^{[2]}),$ where $V^{[2]}$ is the tautological bundle on the Hilbert scheme $S^{[2]}.$ This top Chern class can be computed by using a close formula due to \cite[theorem 4.2]{EGL}, and it turns out to be six, independent of the choice of $V$. See \cite{Joh} for details.    
\end{rmk}

\bibliographystyle{alpha}
\bibliography{references}

\begin{thebibliography}{LLSVS17}

\bibitem[AB12]{AB}
Daniele Arcara and Aaron Bertram.
\newblock Bridgeland-stable moduli spaces for $ k $-trivial surfaces.
\newblock {\em Journal of the European Mathematical Society}, 15(1):1--38,
  2012.

\bibitem[AP06]{AP}
Dan Abramovich and Alexander Polishchuk.
\newblock Sheaves of t-structures and valuative criteria for stable complexes.
\newblock {\em Journal fur die reine und angewandte Mathematik (Crelles
  Journal)}, 2006(590):89--130, 2006.

\bibitem[AT14]{AT}
Nicolas Addington and Richard Thomas.
\newblock Hodge theory and derived categories of cubic fourfolds.
\newblock {\em Duke Mathematical Journal}, 163(10):1885--1927, 2014.

\bibitem[BBD]{BBD}
AA~Beilinson, Joseph Bernstein, and Pierre Deligne.
\newblock Faisceaux pervers, analysis and topology on singular spaces, i
  (luminy, 1981), 5--171.
\newblock {\em Ast{\'e}risque}, 100.

\bibitem[BLMS17]{BLMS}
Arend Bayer, Mart{\'\i} Lahoz, Emanuele Macr{\`\i}, and Paolo Stellari.
\newblock Stability conditions on kuznetsov components.
\newblock {\em arXiv preprint arXiv:1703.10839}, 2017.

\bibitem[Bri07]{Bri}
Tom Bridgeland.
\newblock Stability conditions on triangulated categories.
\newblock {\em Annals of Mathematics}, 166(2):317--345, 2007.

\bibitem[BV88]{BV}
Winfried Bruns and Udo Vetter.
\newblock {\em Determinantal Rings}, volume 1327.
\newblock Springer-Verlag Berlin Heidelberg, 1988.

\bibitem[Deb13]{deb}
Olivier Debarre.
\newblock {\em Curves of Low Degrees on Fano Varieties}, pages 133--145.
\newblock Springer New York, New York, NY, 2013.

\bibitem[EGL99]{EGL}
Geir Ellingsrud, Lothar Göttsche, and M~Lehn.
\newblock On the cobordism class of the hilbert scheme of a surface.
\newblock {\em J. Algebraic Geom}, 10(1), 1999.

\bibitem[Has00]{Ha}
Brendan Hassett.
\newblock Special cubic fourfolds.
\newblock {\em Compositio Mathematica}, 120(1):1–23, 2000.

\bibitem[Joh17]{Joh}
Drew Johnson.
\newblock {Universal Series for Hilbert Schemes and Strange Duality}.
\newblock {\em ArXiv e-prints}, August 2017.

\bibitem[Kuz10]{Kuz4}
Alexander Kuznetsov.
\newblock Derived categories of cubic fourfolds.
\newblock In {\em Cohomological and geometric approaches to rationality
  problems}, pages 219--243. Springer, 2010.

\bibitem[Kuz11]{Kuz}
Alexander Kuznetsov.
\newblock Base change for semiorthogonal decompositions.
\newblock {\em Compositio Mathematica}, 147(3):852--876, 2011.

\bibitem[Kuz17]{Kuz2}
Alexander Kuznetsov.
\newblock Calabi--yau and fractional calabi--yau categories.
\newblock {\em Journal f{\"u}r die reine und angewandte Mathematik (Crelles
  Journal)}, 2017.

\bibitem[Lan83]{lange}
Herbert Lange.
\newblock Universal families of extensions.
\newblock {\em Journal of Algebra}, 83(1):101 -- 112, 1983.

\bibitem[Lie05]{Lie}
Max Lieblich.
\newblock Moduli of complexes on a proper morphism.
\newblock {\em arXiv preprint math/0502198}, 2005.

\bibitem[LK79]{LK}
Knud Lonsted and Steven~L. Kleiman.
\newblock Basics on families of hyperelliptic curves.
\newblock {\em Compositio Mathematica}, 38(1):83--111, 1979.

\bibitem[LLMS17]{LLMS}
Mart{\'\i} Lahoz, Manfred Lehn, Emanuele Macr{\`\i}, and Paolo Stellari.
\newblock Generalized twisted cubics on a cubic fourfold as a moduli space of
  stable objects.
\newblock {\em Journal de Math{\'e}matiques Pures et Appliqu{\'e}es}, 2017.

\bibitem[LLSVS17]{LLSS}
Christian Lehn, Manfred Lehn, Christoph Sorger, and Duco Van~Straten.
\newblock Twisted cubics on cubic fourfolds.
\newblock {\em Journal f{\"u}r die reine und angewandte Mathematik (Crelles
  Journal)}, 2017(731):87--128, 2017.

\bibitem[LP80]{LP}
Eduard {Looijenga} and Chris {Peters}.
\newblock {Torelli theorems for Kaehler K3 surfaces.}
\newblock {\em {Compos. Math.}}, 42:145--186, 1980.

\bibitem[LPZ18]{LPZ}
Chunyi Li, Laura Pertusi, and Xiaolei Zhao.
\newblock Twisted cubics on cubic fourfolds and stability conditions.
\newblock {\em arXiv preprint arXiv:1802.01134}, 2018.

\bibitem[Mur17]{Mur}
Giosu{\`e}~Emanuele Muratore.
\newblock The indeterminacy locus of the {V}oisin map.
\newblock {\em arXiv preprint arXiv:1711.06218}, 2017.

\bibitem[Pol07]{Po}
Alexander~Evgen'evich Polishchuk.
\newblock Constant families of t-structures on derived categories of coherent
  sheaves.
\newblock {\em Moscow Mathematical Journal}, 7(1):109--134, 2007.

\bibitem[Voi16]{Voi}
Claire Voisin.
\newblock {\em Remarks And Questions On Coisotropic Subvarieties and 0-Cycles
  of Hyper-K{\"a}hler Varieties}.
\newblock Springer International Publishing, 2016.

\end{thebibliography}

\end{document}